\documentclass[11pt,a4paper]{article}
\usepackage{epsf,epsfig,amsfonts,amsgen,amsmath,amssymb,amstext,amsbsy,amsopn,amsthm,lineno}
\usepackage{color}
\usepackage{cite}
\usepackage{subfig}
\usepackage{float}
\usepackage{bm}
\usepackage{graphicx,tikz}
\usepackage[colorlinks=true,citecolor=black,linkcolor=black,urlcolor=black]{hyperref}
\hypersetup{colorlinks,
    linkcolor=blue, 
    anchorcolor=blue,
    citecolor=blue}

\newtheorem{theorem}{Theorem}[section]

\newtheorem{lemma}{Lemma}[section]
\newtheorem{remark}{Remark}[section]

\newtheorem{cor}{Corollary}[section]

\theoremstyle{definition}

\newtheorem{claim}{Claim}

\newtheorem{case}{Case}

\tikzstyle{vertex}=[circle, draw, inner sep=0pt, minimum size=3pt]

\numberwithin{equation}{section}
\allowdisplaybreaks

\def\qed{\hfill$\Box$\vspace{12pt}}

\begin{document}
\title
{\bf\Large Signless Laplacian spectral conditions for rainbow matchings in a collection of bipartite graphs\thanks{Supported by the National Natural Science Foundation of China (Nos.~12301456, 12471334), the Support Program for Outstanding Young Talents in Shaanxi Universities (No. 202120009), the Graduate Research and Innovation Project of Yan'an University (YKY2026033), and the Doctoral Research Foundation of Yan'an University (No. YDBK2021-03).}
}

\date{}
\author{
\small Zhiwei Guo\thanks{Corresponding author.\newline
\small E-mail addresses: zhiweiguo@yau.edu.cn (Z.W. Guo),
\small nanxipan@163.com (N.X. Pan), nwlily@yau.edu.cn (L. Li),
\small yychenmath@163.com (Y.Y. Chen).},\quad Nanxi Pan,\quad Li Li,\quad \small Yangyang Chen\\[2mm]
\small School of Mathematics and Computer Science, Yan'an
\small University,\\
\small Yan'an, Shaanxi 716000, PR~China\\[0.3cm]}
\maketitle

\begin{abstract}
Let ${\cal G}=\{G_1,\ldots,G_k\}$ be a collection of (not necessarily distinct) bipartite graphs on the same vertex bipartition $(X,Y)$, where $|X|=a$, $|Y|=b$ and $2\le k\le a\le b$. A \emph{rainbow matching} of ${\cal G}$ is a set of pairwise disjoint edges that can be chosen from distinct members of ${\cal G}$. Denote by $q(G)$ the signless Laplacian spectral radius of a graph $G$. In this paper, we prove that if $q(G_i)\ge b+k-1$ for each $i\in\{1,2,\ldots,k\}$, then ${\cal G}$ admits a rainbow matching of size $k$ unless $G_1=\cdots=G_k\cong K_{k-1,b}\cup\overline{K_{a-k+1}}$, and show that the threshold is sharp and attained by the exceptional collection. The condition is also extended to larger collections for a prescribed level $t$. In addition, we obtain a lower bound for the rainbow matching number in terms of the ordered signless Laplacian spectral radii of the members, and provide a stability version of the extremal characterization. In the proofs, we use the shifting technique and a quotient matrix arising from an equitable partition of a signless Laplacian matrix.
\medskip

\smallskip
\end{abstract}

\section{Introduction}
Throughout this paper, we only consider finite, simple and undirected graphs. For terminology and notation not defined here, we refer the reader to Bondy and Murty \cite{Bondy}.

For a graph $G$, let $V(G)$ and $E(G)$ denote the vertex set and edge set of $G$, respectively, and let $\nu(G)$ denote the matching number of $G$, i.e., the size of the maximum matching of $G$. We use $K_n$ and $\overline{K_n}$ to denote the complete graph and the empty graph on $n$ vertices, respectively, and we regard $\overline{K_0}$ as the empty graph. The \emph{union} of graphs ${G_1}$ and ${G_2}$, denoted by ${G_1}\cup {G_2}$, is a graph with vertex set ${V({G_1})}\cup {V({G_2})}$ and edge set ${E({G_1})}\cup {E({G_2})}$. The \emph{join} of disjoint graphs ${G_1}$ and ${G_2}$, denoted by ${G_1}\vee {G_2}$, is a graph with vertex set ${V({G_1})}\cup {V({G_2})}$ and edge set consisting of $E({G_1})$, $E({G_2})$ and all the edges joining every vertex of ${G_1}$ to every vertex of ${G_2}$. Two graphs ${G_1}$ and ${G_2}$ are said to be \emph{isomorphic}, denoted by
${G_1}\cong{G_2}$, if there exists a bijection $f:V({G_1})\to V({G_2})$ such that $uv\in E({G_1})$ if and only if $f(u)f(v)\in E({G_2})$. In particular, we write ${G_1}={G_2}$ if $V({G_1})=V({G_2})$ and $E({G_1})=E({G_2})$. Let $a$ and $b$ be positive integers, and let $K_{a,b}$ denote the complete bipartite graph on vertex bipartition $(X, Y)$ with $|X|=a$ and $|Y|=b$. For positive integers $n$, $a$ and $b$, we denote by $[n]$ and $[a,b]$ the sets $\{{1, 2,\ldots, n}\}$ and $\{{a,\ldots, b}\}$, respectively.

The \emph{adjacency matrix} of a graph $G$ is defined to be $A(G)={({a_{ij}})_{n\times n}}$, where ${a_{i,j}}=1$ if and only if $ij\in E(G)$ and ${a_{i,j}}=0$ otherwise. The \emph{adjacency spectral radius} $\rho(G)$ of a graph $G$ refers to the maximum eigenvalue of $A(G)$ of $G$. The \emph{signless Laplacian matrix} of a graph $G$ is defined to be $Q(G)=D(G)+A(G)$, where $D(G)=diag({{d_1}, \ldots, {d_n}})$ is the degree diagonal matrix of $G$ and $A(G)$ is the adjacency matrix of $G$. The \emph{signless Laplacian spectral radius} (also known as the $Q$-spectral radius) $q(G)$ of a graph $G$ is defined to be the maximum eigenvalue of the signless Laplacian matrix $Q(G)$ of $G$. Thus, the signless Laplacian matrix $Q(G)$ encodes the vertex degree information of $G$. Let $G$ be a bipartite graph on vertex bipartition $(X, Y)$, and let $R$ be the diagonal matrix with entries $1$ on $X$ and $-1$ on $Y$. Then $RQ(G)R=L(G)=D(G)-A(G)$, where $L(G)$ denotes the Laplacian matrix of $G$. This implies that $Q(G)$ and $L(G)$ are similar matrices, and that $q(G)$ is also the largest Laplacian eigenvalue of $G$.

Let ${\cal G}=\{G_1,\ldots,G_m\}$ be a collection of not necessarily distinct bipartite graphs on the same vertex bipartition $(X,Y)$. Since the indices are retained in the collection, the equal graphs (if exist) are treated as distinct members of the collection. A matching $M=\{e_1,\ldots,e_r\}$ is called a \emph{rainbow matching} of ${\cal G}$ if its edges can be chosen from pairwise distinct members of ${\cal G}$; equivalently, there are distinct indices $i_1,\ldots,i_r\in[m]$ such that $e_j\in E(G_{i_j})$ for each $j\in[r]$. The size of the maximum rainbow matching of ${\cal G}$ is called the \emph{rainbow matching number} of ${\cal G}$, denoted by $\nu_r({\cal G})$.

The existence of rainbow structures in collections of graphs has aroused extensive research interest. In 2020, Aharoni et al. \cite{Aharoni} proved a rainbow version of Mantel's Theorem by considering the existence of rainbow triangles in a collection of graphs. Aharoni et al. \cite{Aharoni} also proposed a conjecture on the existence of rainbow Hamilton cycles in a collection of graphs. In 2021, Cheng et al. \cite{Cheng-Wang} proved that the conjecture due to Aharoni et al. \cite{Aharoni} asymptotically hold. Later, Joos and Kim \cite{Joos} completely verified the conjecture due to Aharoni et al. \cite{Aharoni}, and established the rainbow version of Dirac's Theorem. Li et al. \cite{Li-Li} proved the existence of rainbow spanning trees, rainbow Hamilton paths, rainbow vertex-pancyclicity and other rainbow structures in a collection of graphs, under the Ore-type conditions and Dirac-type conditions. Cheng et al. \cite{Cheng-Sun} further proposed a minimum degree condition for the existence of rainbow Hamilton paths in a collection of graphs, improving the conditions presented by Li \cite{Li-Li}. For more results, we refer the readers to the survey \cite{Sun}.

On the other hand, the spectral conditions for the existence of rainbow structures in a collection of graphs have also been established. Guo et al. \cite{Guo} proved the existence of rainbow matchings in a collection of graphs, under the adjacency spectral conditions. Subsequently, He et al. \cite{He} proved the existence of rainbow Hamilton paths and rainbow linear forests in a collection of graphs, under the adjacency spectral conditions. Later, Zhang et al. \cite{Y.Zhang} proved the existence of rainbow Hamilton cycles in a collection of graphs, under the adjacency spectral conditions and the signless Laplacian spectral conditions. Zhang et al. \cite{X.Zhang} and Wang et al. \cite{Wang} independently proposed signless Laplacian spectral conditions to guarantee the existence of rainbow matchings and Hamilton paths in a collection of graphs. Zhang et al. \cite{L.Zhang} presented adjacency spectral conditions for the existence of rainbow $k$-factors in a collection of graphs. Later, Guo et al. \cite{Guo-Pan} established signless Laplacian spectral conditions for the existence of rainbow $k$-factors in a collection of graphs. Recently, Shi et al. \cite{Shi} gave adjacency spectral conditions for the existence of rainbow matchings in collections of bipartite graphs. Chen et al. \cite{Chen} established adjacency spectral conditions for rainbow Hamilton cycles in collections of bipartite graphs.

In this paper, we establish the best possible signless Laplacian spectral conditions for the existence of rainbow matchings of given size in a collection of bipartite graphs, and we extend the conditions to larger collections for a prescribed level $t$. In addition, we obtain a lower bound for the rainbow matching number in terms of the ordered signless Laplacian spectral radii of the members, and provide a stability version of the extremal characterization, showing that near-extremal signless Laplacian spectral radii force the collection to be close to the extremal collection. In the proofs, we use the shifting technique and a quotient matrix arising from an equitable partition of a signless Laplacian matrix.

\section{Main results}

We first establish signless Laplacian spectral conditions for the existence of rainbow matchings of given size in a collection of bipartite graphs, as follows. Let $H_k=K_{k-1,b}\cup\overline{K_{a-k+1}}$ (see Figure~\ref{fig:extremal}).

\begin{theorem}
\label{Thm1b}
Let $k,a,b$ be positive integers with $2\le k\le a\le b$, and let ${\cal G}=\{G_1,\ldots,G_k\}$ be a collection of bipartite graphs on the same vertex bipartition $(X,Y)$ with $|X|=a$ and $|Y|=b$. If $q(G_i)\ge b+k-1$ for each $i\in[k]$, then ${\cal G}$ admits a rainbow matching of size $k$ unless $G_1=\cdots=G_k\cong H_k$.
\end{theorem}

\begin{remark}
The threshold in Theorem~\ref{Thm1b} is sharp. We can verify that a collection consisting of $k$ copies of $H_k$ contains no rainbow matching of size $k$, and that $q(H_k)=b+k-1$ and $\nu(H_k)=k-1$.
\end{remark}

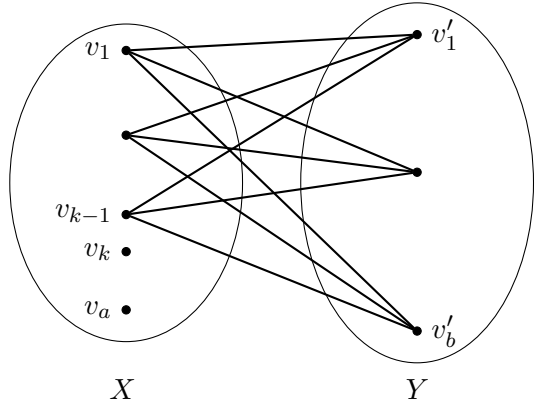
\begin{figure}[htbp]
\begin{center}
\begin{tikzpicture}[scale=0.7,auto,swap]
\tikzstyle{blackvertex}=[circle,draw=black,fill=black]
  \node[label=left: $v_1$,blackvertex,scale=0.3] (a1) at (-4.0,5.5) {};
  \node[blackvertex,scale=0.3] (a2) at (-4.0,3.9) {};
  \node[label=left: $v_{k-1}$,blackvertex,scale=0.3] (a3) at (-4.0,2.4) {};
 \node[label=left: $v_k$,blackvertex,scale=0.3] (a4) at (-4.0,1.7) {};
 \node[label=left: $v_a$,blackvertex,scale=0.3] (a5) at (-4.0,0.6) {};

 \node[label=right: $v_1'$,blackvertex,scale=0.3] (a6) at (1.5,5.8) {};
 \node[blackvertex,scale=0.3] (a7) at (1.5,3.2) {};
 \node[label=right: $v_b'$,blackvertex,scale=0.3] (a8) at (1.5,0.2) {};

 \node [label=right: {$X$}] at (-4.7,-0.9) {};
 \node [label=right: {$Y$}] at (0.9,-0.9) {};
 \draw  (-4,3) ellipse (2.2 and 3);
 \draw  (1.5,3) ellipse (2.2 and 3.4);
 \draw [black,thick] (a1) -- (a6);
 \draw [black,thick] (a1) -- (a7);
 \draw [black,thick] (a1) -- (a8);
 \draw [black,thick] (a2) -- (a6);
 \draw [black,thick] (a2) -- (a7);
 \draw [black,thick] (a2) -- (a8);
 \draw [black,thick] (a3) -- (a6);
 \draw [black,thick] (a3) -- (a7);
 \draw [black,thick] (a3) -- (a8);
\end{tikzpicture}
\caption{The graph $H_k$.}
\label{fig:extremal}
\end{center}
\end{figure}

From Theorem \ref{Thm1b}, by setting $k=a=b=n$, the following corollary is immediate.

\begin{cor}
\label{cor1b}
Let $n$ be a positive integer, and let ${\cal G}=\{G_1,\ldots,G_n\}$ be a collection of bipartite graphs on the same vertex bipartition $(X,Y)$ with $|X|=|Y|=n$. If $q(G_i)\ge2n-1$ for each $i\in[n]$, then ${\cal G}$ admits a rainbow perfect matching unless $G_1=\cdots=G_n\cong K_{n-1,n}\cup\overline{K_1}$.
\end{cor}

Let $a,b,t$ be positive integers with $2\le t\le a\le b$, and let $\mathfrak E_{a,b,t}$ denote the set of all bipartite graphs on the same vertex bipartition $(X,Y)$ that are isomorphic to $K_{t-1,b}\cup\overline{K_{a-t+1}}$. We next extend Theorem~\ref{Thm1b} to larger collections for a prescribed level $t$, as follows.

\begin{theorem}
\label{Thm-high-members}
Let ${\cal G}=\{G_1,\ldots,G_m\}$ be a collection of bipartite graphs on the same vertex bipartition $(X,Y)$ with $|X|=a\le b=|Y|$, and let $t$ be an integer with $2\le t\le\min\{m,a\}$. If there exist at least $t$ graphs $G$ in ${\cal G}$  such that $q(G)\ge b+t-1$, then either ${\cal G}$ admits a rainbow matching of size $t$, or $\nu_r({\cal G})=t-1$ and there is $H\in\mathfrak E_{a,b,t}$ such that $E(G_i)\subseteq E(H)$ for each $i\in[m]$ and $G_i=H$ whenever $q(G_i)\ge b+t-1$.
\end{theorem}

For a collection ${\cal G}=\{G_1,\ldots,G_m\}$ of bipartite graphs on the same vertex bipartition $(X,Y)$ with $|X|=a\le b=|Y|$, let $q_1^\downarrow({\cal G})\ge\cdots\ge q_m^\downarrow({\cal G})$ denote the signless Laplacian spectral radii
of the members of ${\cal G}$ arranged in nonincreasing order, and let $T=\max\left(\{1\}\cup
\{t:2\le t\le s,\ q_t^\downarrow(G)\ge b+t-1\}\right)$, where $s=\min\{m,a\}$. The following Theorem~\ref{thm-quant} establishes a lower bound for the rainbow matching number in terms of the ordered signless Laplacian spectral radii of the members.
\begin{theorem}
\label{thm-quant}
Let ${\cal G}=\{G_1,\ldots,G_m\}$ be a collection of bipartite graphs on the same vertex bipartition $(X,Y)$ with $|X|=a\le b=|Y|$. Then $\nu_r({\cal G})\ge T-1$ and the equality holds if and only if there exists $H\in\mathfrak E_{a,b,T}$ such that $E(G_i)\subseteq E(H)$ for each $i\in[m]$ and at least $T$ members are equal to $H$.
\end{theorem}

For positive integers $a,b,m,t$, let $\mathfrak C_t$ be the set of collections ${\cal F}=\{F_1,\ldots,F_m\}$ such that for some $H\in\mathfrak E_{a,b,t}$ and each $i\in[m]$, $F_i$ is a spanning subgraph of $H$ and at least $t$ members of ${\cal F}$ are equal to $H$, and define
\[
d_\triangle({\cal G},\mathfrak C_t)
=\min_{{\cal F}\in\mathfrak C_t}\max_{i\in[m]}
\frac{|E(G_i)\triangle E(F_i)|}{ab},
\]
where corresponding members retain their indices. The following Corollary~\ref{cor-finite-gap} establishes finite extremal gap.

\begin{cor}
\label{cor-finite-gap}
Let ${\cal G}=\{G_1,\ldots,G_m\}$ be a collection of bipartite graphs on the same vertex bipartition $(X,Y)$ with $|X|=a\le b=|Y|$, and let $t$ be an integer with $2\le t\le a\le b$ and $m\ge t$. Then for every $\varepsilon>0$ there exists $\Delta=\Delta(a,b,m,t,\varepsilon)>0$ such that if $\nu_r({\cal G})<t$ and $q_t^\downarrow({\cal G})\ge b+t-1-\Delta$, then $d_\triangle({\cal G},\mathfrak C_t)<\varepsilon$.
\end{cor}

\section{Preliminaries}
In this section, we introduce the shifting technique, and present several lemmas that will be used in the proofs of the main results.

The shifting technique\cite{Frankl1}, also known as the Kelmans operation \cite{Kelmans}, has been found to be extremely effective in determining the upper bounds on the spectral radius of a graph (see\cite{Frankl2}). The original definition of the shifting technique goes back to the paper of Erd\H{o}s, Ko and Rado \cite{Erdos}. For a graph $G$ on vertex set $[n]$ and any two vertices $x, y\in[n]$ with $x\ne y$, we define the \emph{$(x,y)$-shift} of $G$ as ${S_{xy}}(G)=\{{{S_{xy}}(e)|~e \in E(G)}\},$ where
\[S_{xy}(e)=
\begin{cases}(e\setminus\{y\})\cup\{x\},&\text{if }y\in e,x\notin e\text{ and } (e\setminus\{y\})\cup\{x\}\notin E(G); \\
e,&\text{otherwise}.
\end{cases}\]

One can verify that the $\left( {x,y} \right)$-shift ${S_{xy}}(G)$ only attempts to modify the edges of $G$ containing no $x$ but $y$. For the $\left( {x,y} \right)$-shift ${S_{xy}}(G)$, we call $y$ the source and $x$ the target of the shift. In the remaining part of this paper, we also use the notation ${S_{xy}}(G)$ to denote the resulting graph obtained from a given graph $G$ by performing the $(x, y)$-shift ${S_{xy}}(G)$. Clearly, we can conclude that $|E(S_{xy}(G))|=|E(G)|$ by the definition of ${S_{xy}}(G)$. The resulting graph $S(G)$ obtained from a given graph $G$ by performing $(x,y)$-shifts for all vertex pairs $(x, y)$ with $x<y$ is called the \emph{shifted graph} of $G$. A graph $G$ is said to be \emph{shifted} if $S(G)=G$. In 2016, Li and Ning \cite{Li-Ning} proved that the shifting technique does not decrease the signless Laplacian spectral radius of a graph $G$, as follows.

\begin{lemma}[Li and Ning \cite{Li-Ning}]\label{lemma1b}
Let $x$, $y$ be two vertices of $G$ with $x\ne y$. Then $q({S_{xy}(G)})\ge q(G)$.
\end{lemma}

In 2025, Zhang et al. \cite{X.Zhang} proved that the above inequality is strict provided that the graph $G$ is connected and that $G$ and ${S_{xy}}(G)$ are not isomorphic, as follows.

\begin{lemma}[Zhang et al. \cite{X.Zhang}]\label{lemma2b}
Let $G$ be a connected graph on vertex set $[n]$, and let $x$, $y$ be two vertices of $G$ with $x\ne y$. Then $q({S_{xy}(G)})>q(G)$ unless $G\cong{S_{xy}}(G)$.
\end{lemma}

For a bipartite graph on the vertex bipartition $(X,Y)$, the two vertices of an $(x,y)$-shift are always chosen from the same part so that the resulting graph remains bipartite. Since each $(x, y)$-shift strictly reduces the sum of the vertex elements of edges of $G$, the following Lemma \ref{lemma3b} is immediate by the definition of the shifted graph.

\begin{lemma}\label{lemma3b}
Let $G$ be a shifted graph. Then $\{{{y_1},{y_2}}\}\in E(G)$ implies $\{ {{x_1},{x_2}}\} \in E(G)$ for each $\{{{x_1},{x_2}}\},\{{{y_1},{y_2}}\} \subseteq [n]$ such that ${x_i}\le {y_i}$ for $i \in \{1,2\}$.
\end{lemma}

For a given collection ${\cal G}=\{{G_1}, \ldots, {G_m}\}$ and two vertices $x,y\in[n]$ with $x \ne y$, set ${S_{xy}}({\cal G})=\{{S_{xy}}({{G_1}}),\ldots,$ ${S_{xy}}({G_m})\}$ and $S({\cal G})=\{S({G_1}),\ldots ,S({G_m})\} $. In 2025, Huang et al. \cite{Huang} proved that the shifting technique preserves the existence of rainbow matchings in the following sense.

\begin{lemma}[Huang et al. \cite{Huang}]\label{lemma4b}
Let ${\cal G}=\{{G_1}, \ldots ,{G_m}\} $ be a collection of graphs on the same vertex set $[n]$. If $S({\cal G})$ admits a rainbow matching, then so does ${\cal G}$.
\end{lemma}

A component is \emph{nontrivial} if it contains at least one edge. We next establish an equality characterization for shifting.

\begin{lemma}\label{lemma-equality-shift}
Let $G^*$ be a graph obtained from a graph $G$ by a finite sequence of shifts. If $G^*$ has exactly one nontrivial component and $q(G)=q(G^*)>0$, then $G\cong G^*$.
\end{lemma}

\begin{proof}
Let $G=F_0,F_1,\ldots,F_\ell=G^*$ be the shift sequence. By Lemma~\ref{lemma1b}, we deduce that $q(F_0)\le q(F_1)\le\cdots\le q(F_\ell)$. Since \(q(F_0)=q(F_\ell)\), we have $q(F_0)=q(F_1)=\cdots=q(F_\ell)$.

Suppose that some \(F_j\) has at least two nontrivial components. Since the signless Laplacian matrix of a graph is block diagonal with respect to its connected components, we have \(q(F_j)=\max\{q(C): C~\text{is a component of} F_j\}\). Choose a component \(C_j\) of \(F_j\) such that \(q(C_j)=q(F_j)\). We track the image of the component $C_j$ through the subsequent
shifts. At each step, we denote by $C_i$ the component obtained from $C_j$ after applying the shifts up to $F_i$, provided that no shift has merged it with another nontrivial component.

Before the first shift that joins the tracked component with another nontrivial component, a shift either does not affect the tracked component, involves it and an isolated vertex, or is performed within it. The first two cases preserve its
isomorphism type, because a shift involving an isolated vertex only changes the label of a vertex of the component. In the last case, let $C'$ be the component obtained after the shift. By Lemma~\ref{lemma2b}, if $C'\not\cong C_j$, then
$q(C')>q(C_j)$. Since $q(C_j)=q(F_i)$ and $C'$ is a component of $F_{i+1}$, we have $q(F_{i+1})\ge q(C')>q(F_i)$, a contradiction. Hence, the tracked component remains isomorphic before the first joining shift.

Let $F_i$ be the graph immediately before this first joining shift, and let $C_i$ be the tracked component in $F_i$. Let \(x\in V(C)\) be the target vertex, and let \(y\) be the source vertex in another nontrivial component. Consider the shift \(S_{xy}\). The two possible orientations of this shift are isomorphic, and hence have the same signless Laplacian spectral radius.

Let \(z\) be a unit Perron vector of \(Q(C)\). We extend \(z\) by assigning zero entries to all vertices outside the vertex set of $C$. Since \(C\) is connected and nontrivial, \(z_v>0\) for every \(v\in V(C)\); in particular, \(z_x>0\).

One can check that every neighbor of \(y\) lies outside \(C\) and is nonadjacent to \(x\). Since the entries of $z$ outside $C$ are zero, each edge moved from $yu$ to $xu$ increases the quadratic form by $(z_x+z_u)^2=z_x^2$. Hence, by
$z^{T}Q(F)z=\sum_{uv\in E(F)}(z_u+z_v)^2$, we obtain $z^{T}Q(S_{xy}(F))z=z^{T}Q(F)z+d_F(y)z_x^2$. Since $d_F(y)\ge1$ and $z_x>0$, we have $z^{T}Q(S_{xy}(F))z>z^{T}Q(F)z=q(F)$, and hence $q(S_{xy}(F))>q(F)$, a contradiction. Thus, every graph in the shift sequence has exactly one nontrivial component.

We now show that every shift in the sequence preserves the isomorphism type. A shift involving an isolated vertex only relabels the component, while a shift inside the component preserves its isomorphism type by Lemma~\ref{lemma2b} and the equality of the spectral radii. A shift between two isolated vertices has no effect. Therefore, \(F_0\cong F_1\cong\cdots\cong F_\ell\), and hence
\(G\cong G^*\). \qed
\end{proof}

Next, we recall the equitable partition of a real symmetric matrix, which will be used to calculate signless Laplacian spectral radii. Let $M$ be a real symmetric matrix whose rows and columns are indexed by a set $V$, and let $\pi=\{V_1,\ldots,V_r\}$ be a partition of $V$. Write $M_{ij}$ for the submatrix with rows in $V_i$ and columns in $V_j$. The partition $\pi$ is \emph{equitable} if every row of $M_{ij}$ has the same sum, say $b_{ij}$, for each $i,j$. The matrix $(b_{ij})_{r\times r}$ is the \emph{quotient matrix} of $M$ with respect to $\pi$, denoted by $M/\pi$.

The following lemma indicates that the equitable partition can be used to simplify the process of calculating the signless Laplacian spectral radius of a graph.

\begin{lemma}[Godsil and Royle \cite{Godsil}]\label{lemma5b}
Let $G$ be a graph. If $\pi$ is an equitable partition of $V(G)$ corresponding to $Q(G)$, then $\lambda_{\max}(Q(G)/\pi)=q(G)$.
\end{lemma}

We also use the following monotonicity result of Brouwer and Haemers \cite{Brouwer}.

\begin{lemma}[Brouwer and Haemers \cite{Brouwer}]\label{lemma6b}
Let $H$ be a subgraph of a connected graph $G$. Then $q(H)\le q(G)$ and the equality holds if and only if $H \cong G$.
\end{lemma}

\begin{lemma}\label{lemma7b}
Let $r,s,t$ be positive integers, and let $H=K_{s,t}\cup\overline{K_r}$. Then $q(H)=s+t$, and $q(F)<s+t$ whenever $F$ is a proper spanning subgraph of $H$.
\end{lemma}

\begin{proof}
The vertices in $\overline{K_r}$ are isolated and contribute only zero eigenvalues to $Q(H)$. The two parts of $K_{s,t}$ form an equitable partition of $Q(K_{s,t})$ with quotient matrix
\[
\begin{pmatrix}
t&t\\
s&s
\end{pmatrix}.
\]
Since the quotient matrix has eigenvalues $0$ and $s+t$, we have $q(H)=s+t$ by Lemma~\ref{lemma5b}.

Let $F$ be a proper spanning subgraph of $H$. The vertices in $\overline{K_r}$ remain isolated in $F$. By deleting these isolated vertices, we obtain a proper spanning subgraph $F_0$ of $K_{s,t}$. Since the vertices in $\overline{K_r}$ contribute only zero blocks to the signless Laplacian matrix, $Q(F)$ is the direct sum of $Q(F_0)$ and an $r\times r$ zero matrix, and hence $q(F)=q(F_0)$. Since $K_{s,t}$ is connected, by applying Lemma~\ref{lemma6b} to the proper spanning subgraph $F_0$ of $K_{s,t}$, we conclude that $q(F_0)<q(K_{s,t})=s+t$. Since $q(F)=q(F_0)$, we have $q(F)<q(K_{s,t})=s+t$. \qed
\end{proof}

\section{Proofs of the main results}
In this section, we prove the main results by using the shifting technique.

\noindent \textbf{Proof of Theorem~\ref{Thm1b}.}
Let ${\cal G}=\{G_1,\ldots,G_k\}$ satisfy the hypotheses of Theorem~\ref{Thm1b}. Suppose to the contrary that ${\cal G}$ contains no rainbow matching of size $k$. We relabel the same vertex bipartition as $X=[a]$ and $Y=[b]'$. We apply the shifting technique simultaneously within $X$ and within $Y$, and fix a resulting terminal collection $S({\cal G})=\{S_1,\ldots,S_k\}$. It follows from Lemma~\ref{lemma4b} that $S({\cal G})$ contains no rainbow matching of size $k$. By Lemma~\ref{lemma1b}, we conclude that
\begin{equation}
\label{eq:shift-lower}
q(S_i)\ge q(G_i)\ge b+k-1
\end{equation}
for all $i\in[k]$.

For $j\in[k]$, put $e_j=j(k-j+1)'$. The edges $e_1,\ldots,e_k$ are pairwise disjoint.

\begin{claim}\label{claim1a}
For each $i\in[k]$, the graph $S_i$ contains $e_2,\ldots,e_{k-1}$.
\end{claim}

\begin{proof}
For the case of $k=2$, the assertion is vacuous. We now assume that $k\ge3$. Suppose that $e_r\notin E(S_i)$ for some $i\in[k]$ and $2\le r\le k-1$. If $uv'\in E(S_i)$ with $u\ge r$ and $v'\ge k-r+1$, then by Lemma~\ref{lemma3b}, we deduce that $e_r=r(k-r+1)'\in E(S_i)$, a contradiction. Hence, no edge of $S_i$ joins $\{r,\ldots,a\}$ to $\{(k-r+1)',\ldots,b'\}$.

Let $H_r$ be the graph obtained from $K_{a,b}$ by deleting all edges between these two sets. Since $r-1\ge1$ and $k-r\ge1$, the vertices $1,\ldots,r-1$ are adjacent to all of $Y$. However, the vertices $1',\ldots,(k-r)'$ are adjacent to all of $X$. It follows that $H_r$ is connected. Thus, $S_i$ is a spanning subgraph of $H_r$. By Lemma~\ref{lemma6b}, we have $q(S_i)\le q(H_r)$. Let $X_1=[r-1]$, $X_2=[r,a]$, $Y_1=[1,k-r]'$ and $Y_2=[k-r+1,b]'$. The partition $\pi=\{X_1,X_2,Y_1,Y_2\}$
is equitable with respect to $Q(H_r)$. Indeed, the vertices in \(X_1,X_2,Y_1,Y_2\) have degrees \(b,k-r,a,r-1\), respectively, and their neighbors are distributed among these four sets according to the quotient matrix
\[B=\begin{pmatrix}
b&0&k-r&b-k+r\\
0&k-r&k-r&0\\
r-1&a-r+1&a&0\\
r-1&0&0&r-1
\end{pmatrix},
\]
where the entries of \(B\) are obtained by summing the corresponding rows of \(Q(H_r)\) over the four parts.

Put $p=k-r$, $h=r-1$, and $s=p+h=k-1$, and $d=b-a\ge0$. The characteristic polynomial of $B$ is
$g(x):=\det(xI-B)=x^4-(a+b+t)x^3+(2ph+pb+ah+ab)x^2-ph(a+b)x$. By Lemma~\ref{lemma5b}, $q(H_r)$ is the largest eigenvalue of $B$, and hence $q(H_r)$ is a root of $g$. Substituting $a=b-d$, $t=p+h$, and $x=b+t+y$ gives
$g(b+t+y)=y^4+c_3y^3+c_2y^2+c_1y+c_0$, where
\begin{align*}
c_3={}&2b+d+3h+3p,\\
c_2={}&b^2+2bd+4bh+4bp+2dh+3dp+3h^2+8hp+3p^2,\\
c_1={}&b^2d+b^2h+b^2p+2bdh+4bdp+2bh^2+6bhp+2bp^2\\
&\quad+dh^2+5dhp+3dp^2+h^3+7h^2p+7hp^2+p^3,\\
c_0={}&p(b+h+p)\bigl(bd+2dh+dp+2h^2+2hp\bigr).
\end{align*}
Since $p,h\ge1$, $b>0$, and $d\ge0$, each of the four coefficients is positive. The positivity of these coefficients implies $g(b+s+y)>0$ for every $y\ge0$. Hence, $g(x)>0$ whenever $x\ge b+s$. Since $q(H_r)$ is a root of $g$, and $g(x)>0$ for all
$x\ge b+k-1$, we deduce that $q(H_r)<b+k-1$, contrary to Inequality~\eqref{eq:shift-lower}.\qed
\end{proof}

By Claim~\ref{claim1a}, every $S_i$ contains $e_2,\ldots,e_{k-1}$. We say that an edge \emph{occurs} in $S({\cal G})$ if it belongs to at least one $S_i$; otherwise it is absent from every member. We now consider the following two cases.

\begin{case}
Both $e_1$ and $e_k$ occur in $S({\cal G})$.
\end{case}

Let $I_1:=\{i:e_1\in E(S_i)\}$ and $I_k:=\{i:e_k\in E(S_i)\}$. In this case, both sets are nonempty. If $w\in I_1$ and $w'\in I_k$ are distinct, then we assign $e_1$ to $S_w$, $e_k$ to $S_{w'}$, and $e_2,\ldots,e_{k-1}$ bijectively to the remaining $k-2$ members. The assigned edges are pairwise disjoint and form a rainbow matching of size $k$. Hence, $w=w'$ for every $w\in I_1$ and $w'\in I_k$. Since both sets are nonempty, every member of $I_1\cup I_k$ has the
same index. Hence, $I_1=I_k=\{s\}$ for some $s\in[k]$. After relabelling the members if necessary, we may assume that $s=k$. Thus, neither $e_1$ nor $e_k$ is an edge of $S_i$ for $i\in[k-1]$. Since $e_k=k1'\notin E(S_i)$, by Lemma~\ref{lemma3b}, we conclude that no vertex of $\{k,\ldots,a\}$ is incident with an edge of $S_i$, i.e., every edge $uv'$ with $u\ge k$ would force $k1'\in E(S_i)$. Hence, $S_i\subseteq H_k=K_{k-1,b}\cup\overline{K_{a-k+1}}$.
Since $e_1\in E(H_k)\setminus E(S_i)$, $S_i$ is a proper spanning subgraph of $H_k$. By Lemma~\ref{lemma7b}, we can conclude that $q(G_i)\le q(S_i)<q(H_k)=b+k-1$, contrary to the hypothesis.

\begin{case}
\label{case2b}
At least one of $e_1$ and $e_k$ does not occur in $S({\cal G})$.
\end{case}

Suppose first that $e_1=1k'$ does not occur; equivalently, $e_1$ is absent from every $S_i$. By Lemma~\ref{lemma3b}, every edge $uv'$ with $v\ge k$ would force $1k'\in E(S_i)$. Hence, no vertex of $\{k',\ldots,b'\}$ is incident with an edge of $S_i$, and $S_i\subseteq H_1=K_{a,k-1}\cup\overline{K_{b-k+1}}$. By Lemma~\ref{lemma7b} and \eqref{eq:shift-lower}, we have $b+k-1\le q(G_i)\le q(S_i)\le q(H_1)=a+k-1$. Since $a\le b$, this chain forces $a=b$ and equality throughout. By Lemma~\ref{lemma7b}, we deduce that $S_i=H_1$. If $S_i\neq H_1$, then $S_i$ would be a proper spanning subgraph of
$H_1$. By Lemma~\ref{lemma7b}, we have $q(S_i)<q(H_1)$, a contradiction.

By Lemma~\ref{lemma-equality-shift}, we deduce that $G_i\cong H_1\cong H_k$ for each $i$. Indeed, it follows from \(a=b\) that exchanging the two bipartition classes yields $H_1=K_{a,k-1}\cup\overline{K}_{b-k+1}\cong K_{k-1,b}\cup\overline{K}_{a-k+1}=H_k$. Hence \(G_i\cong H_k\) for each \(i\).

Now, we assume that $e_1$ occurs. In Case~\ref{case2b}, the edge $e_k=k1'$ must be absent from every $S_i$. By Lemma~\ref{lemma3b}, we have $S_i\subseteq H_k$. By Lemma~\ref{lemma7b} and Inequality~ \eqref{eq:shift-lower}, we deduce that $b+k-1\le q(G_i)\le q(S_i)\le q(H_k)=b+k-1$. All the inequalities in the above derivation hold with equality, which implies that $S_i=H_k$. By Lemma~\ref{lemma-equality-shift} and the fact $S_i=H_k$, we have $G_i\cong H_k$ for each $i$.

We now show that the isomorphic copies obtained above coincide on the fixed vertex bipartition. We first consider the case $a<b$. The unique nontrivial component of $H_k$ is $K_{k-1,b}$, and its side of size $b$ must occupy all of $Y$. Hence, every copy of $H_k$ on the fixed vertex bipartition $(X,Y)$ has the form $K_{A_i,Y}$, where $A_i\subseteq X$ and $|A_i|=k-1$. If the sets $A_1,\ldots,A_k$ are not all equal, then $|\bigcup_{i=1}^kA_i|\ge k$. If $I\subseteq[k]$ is nonempty and $I\ne[k]$, then $\bigcup_{i\in I}A_i$ contains one of the sets $A_i$ and hence has size at least $k-1\ge|I|$; for $I=[k]$, the union has size at least $k$. By Hall's theorem, we deduce that distinct representatives $x_i\in A_i$ for $i\in[k]$. Since $b\ge k$, we choose distinct vertices $y_1,\ldots,y_k\in Y$. It follows that $x_iy_i\in E(G_i)$ for all $i$, a contradiction. Hence, $A_1=\cdots=A_k$.

We next consider the case $a=b$. Each $G_i\cong H_k$ is either $K_{A_i,Y}$ with $|A_i|=k-1$, or $K_{X,B_i}$ with $|B_i|=k-1$. Suppose that both types occur. Let $I$ and $J$ index the two types with $|I|=p$ and $|J|=k-p$, where $1\le p\le k-1$. For every nonempty $I'\subseteq I$, $\left|\bigcup_{i\in I'}A_i\right|\ge k-1\ge p\ge|I'|$, so Hall's theorem gives distinct representatives $x_i\in A_i$ for $i\in I$. Similarly, for every nonempty $J'\subseteq J$, $\left|\bigcup_{j\in J'}B_j\right|\ge k-1\ge k-p\ge|J'|$, and we may choose distinct representatives $y_j\in B_j$ for $j\in J$. Since $a=b\ge k$, the set $Y\setminus\{y_j:j\in J\}$ has size $b-(k-p)\ge p$, and the set $X\setminus\{x_i:i\in I\}$ has size $a-p\ge k-p$. We choose distinct $y_i$ for $i\in I$ from the first set and distinct $x_j$ for $j\in J$ from the second. It follows that the edges $x_iy_i$ ($i\in I$) and $x_jy_j$ ($j\in J$) form a rainbow matching of size $k$, a contradiction. Thus, all $G_i$ have the same orientation. The preceding Hall argument, applied to the sets $A_i$ or to the sets $B_i$, shows that all these sets are equal. Therefore, $G_1=\cdots=G_k\cong K_{k-1,b}\cup\overline{K_{a-k+1}}$. This concludes the proof of Theorem \ref{Thm1b}.\qed

\noindent \textbf{Proof of Corollary~\ref{cor1b}.}
By Theorem~\ref{Thm1b} with $k=a=b=n$, we deduce that $b+k-1=2n-1$, and that a matching of size $n$ is perfect. We can conclude that the exceptional graph is $K_{n-1,n}\cup\overline{K_1}$.\qed

\noindent \textbf{Proof of Theorem~\ref{Thm-high-members}.}
Suppose that $\nu_r({\cal G})<t$, and put $I_t:=\{i\in[m]:q(G_i)\ge b+t-1\}$.
By the hypothesis, $|I_t|\ge t$. Let $J$ be an arbitrary $t$-element subset of $I_t$. By Theorem~\ref{Thm1b}, the
members indexed by $J$ coincide with a graph in $E_{a,b,t}$. Since $|I_t|\ge t$, every pair of indices in $I_t$ can be extended to a $t$-element subset of $I_t$. Therefore, all members indexed by $I_t$ coincide with a single graph $H\in E_{a,b,t}$. Since $|I_t|\ge t$, every pair of indices in $I_t$ is contained in a $t$-element subset of $I_t$. Hence, all $G_i$ with $i\in I_t$ coincide with one graph $H\in\mathfrak E_{a,b,t}$.

We claim that every remaining member is a spanning subgraph of $H$. Suppose, to the contrary, that $e\in E(G_j)\setminus E(H)$ for some $j\notin I_t$. If $H=K_{|A|,|Y|}$ with $A\subseteq X$ and $|A|=t-1$, then the $X$-end of $e$ lies in $X\setminus A$. Since $b\ge t$, we choose $t-1$ vertices of $Y$ different from the $Y$-end of $e$ and match them bijectively to the vertices of $A$. Therefore, these edges form a matching of size $t-1$ in $H$ disjoint from $e$. If $a=b$ and $H$ has the opposite orientation, then we interchange the two parts in this construction. We assign these $t-1$ edges to distinct members indexed by $I_t$ and assign $e$ to $G_j$. The resulting rainbow matching has size $t$, a contradiction. Hence, $E(G_j)\subseteq E(H)$ for every $j\in[m]$.

Since every member is a subgraph of $H$, $\nu_r({\cal G})\le\nu(H)=t-1$. Conversely, a matching of size $t-1$ in $H$ may be assigned to $t-1$ distinct members indexed by $I_t$. Hence, $\nu_r({\cal G})=t-1$.
\qed

\noindent \textbf{Proof of Theorem~\ref{thm-quant}.}
If $T=1$, then $\nu_r(G)\ge0=T-1$, and the conclusion is immediate. We now assume that \(T\ge2\). By definition of $T$, at least $T$ members satisfy $q(G_i)\ge b+T-1$. By Theorem~\ref{Thm-high-members}, and setting $t=T$, we conclude that either $\nu_r({\cal G})\ge T$ or the stated exceptional structure with $\nu_r({\cal G})=T-1$. Hence, $\nu_r({\cal G})\ge T-1$, and equality can occur only in the stated exceptional structure.

Conversely, suppose that there is $H\in\mathfrak E_{a,b,T}$ such that $E(G_i)\subseteq E(H)$ for every $i$ and at least $T$ members equal to $H$. Since $\nu(H)=T-1$, every rainbow matching has size at most $T-1$. A matching of size $T-1$ in $H$ can be assigned to distinct members among the $T$ copies of $H$, implying $\nu_r({\cal G})=T-1$.
\qed

\noindent \textbf{Proof of Corollary~\ref{cor-finite-gap}.}
Let $\mathcal B_\varepsilon$ be the set of all collections of $m$ graphs ${\cal F}$ on the fixed vertex bipartition such that $\nu_r({\cal F})<t$ and $d_\triangle({\cal F},\mathfrak C_t)\ge\varepsilon$. Since the bipartition is fixed, there are only finitely many possible graphs on this bipartition, and hence only finitely many collections of
$m$ graphs. Therefore, $B_\varepsilon$ is finite.

If $\mathcal B_\varepsilon=\emptyset$, then every collection satisfying $\nu_r({\cal G})<t$ already has $d_\triangle({\cal G},\mathfrak C_t)<\varepsilon$, and any positive $\Delta$ satisfies the conclusion. Suppose now that $\mathcal B_\varepsilon\ne\emptyset$, and put $M_\varepsilon:=\max_{{\cal F}\in\mathcal B_\varepsilon} q_t^\downarrow({\cal F})$. We claim that $M_\varepsilon<b+t-1$. Otherwise some ${\cal F}\in\mathcal B_\varepsilon$ would have at least $t$ members with signless Laplacian spectral radius at least $b+t-1$. By Theorem~\ref{Thm-high-members}, we deduce that ${\cal F}\in\mathfrak C_t$, contrary to the definition of $\mathcal B_\varepsilon$.

Let $\Delta:=(b+t-1-M_\varepsilon)/2>0$. It follows that $b+t-1-\Delta>M_\varepsilon$. Let ${\cal G}$ satisfy $\nu_r({\cal G})<t$ and $q_t^\downarrow({\cal G})\ge b+t-1-\Delta$.

If $G\in B_\varepsilon$, then by the definition of $M_\varepsilon$, we have $q_t^\downarrow(G)\le M_\varepsilon$, which contradicts $q_t^\downarrow(G)\ge b+t-1-\Delta>M_\varepsilon$. Therefore, $G\notin B_\varepsilon$, and hence   $d_4(G,\mathcal C_t)<\varepsilon$.\qed

\section{Conclusions and final remarks}
This work mainly focuses on the existence of rainbow matchings in a collection of bipartite graphs based on signless Laplacian spectral conditions. In fact, a perfect matching and a rainbow Hamilton cycle can be regarded as special $k$-factors when $k=1$ and $k=2$, where a $k$-factor refers to a spanning subgraph $H$ of $G$ such that ${d_H(v)}=k$ for each $v\in V(G)$. The sufficient condition in terms of the spectral radius for the existence of rainbow Hamilton cycles in a collection of bipartite graphs has been established, as follows.

\begin{theorem}[\cite{Chen}]\label{Thm2b}
Let $n$ be a positive integer with $n\geq2$, and let ${\cal G}=\{{G_1},\ldots,{G_{2n}}\} $ be a collection of balanced bipartite graphs on the same vertex bipartition $(X,Y)$. If $\rho({{G_i}})\ge\rho({{K_{1,n-1}}\cup\overline{{K_{n-1,1}}}})$
for each $i\in [2n]$, then ${\cal G}$ admits a rainbow Hamilton cycle unless ${G_1}={G_2}=\cdots={G_{2n}}\cong{K_{1,n-1}}\cup\overline{{K_{n-1,1}}}$.
\end{theorem}

Inspired by these existing spectral results, it is worth considering to establish signless Laplacian spectral conditions for the existence of rainbow Hamilton cycles and rainbow $k$-factors in a collection of bipartite graphs.

\section*{Data availability}
No data was used for the research described in the paper.

\section*{Declaration of conflicts of interest}
The authors declare that there are no known competing financial interests or personal relationships that could have appeared to influence the work reported in the paper.

\end{document}